\documentclass[12pt]{amsart}

\usepackage{amsmath}
\usepackage{caption}
\usepackage[curve]{xypic}
\usepackage{cite}
\usepackage{enumitem}
\usepackage{color}
\usepackage{url}
\usepackage[usenames,dvipsnames]{xcolor}
\usepackage{graphicx}
\usepackage[all,cmtip]{xy}

\makeatletter 
\def\@cite#1#2{{\m@th\upshape\bfseries%
[{#1\if@tempswa{\m@th\upshape\mdseries, #2}\fi}]}}
\makeatother 

\theoremstyle{plain}
\newtheorem{thm}{Theorem}[section]
\newtheorem{cor}[thm]{Corollary}

\newtheorem{lem}[thm]{Lemma}

\theoremstyle{definition}

\theoremstyle{remark}
\newtheorem{rem}[thm]{Remark}

\numberwithin{equation}{subsection}
\renewcommand{\bold}[1]{\medskip \noindent {\bf #1 }\nopagebreak}

\newcommand{\nc}{\newcommand}
\newcommand{\rnc}{\renewcommand}

\nc\bA{\mathbb{A}}
\nc\bB{\mathbb{B}}
\nc\bC{\mathbb{C}}
\nc\bD{\mathbb{D}}
\nc\bE{\mathbb{E}}
\nc\bF{\mathbb{F}}
\nc\bG{\mathbb{G}}
\nc\bH{\mathbb{H}}
\nc\bI{\mathbb{I}}
\nc{\bJ}{\mathbb{J}} 
\nc\bK{\mathbb{K}}
\nc\bL{\mathbb{L}}
\nc\bM{\mathbb{M}}
\nc\bN{\mathbb{N}}
\nc\bO{\mathbb{O}}
\nc\bP{\mathbb{P}}
\nc\bQ{\mathbb{Q}}
\nc\bR{\mathbb{R}}
\nc\bS{\mathbb{S}}
\nc\bT{\mathbb{T}}
\nc\bU{\mathbb{U}}
\nc\bV{\mathbb{V}}
\nc\bW{\mathbb{W}}
\nc\bY{\mathbb{Y}}
\nc\bX{\mathbb{X}}
\nc\bZ{\mathbb{Z}}
\nc\cA{\mathcal{A}}
\nc\cB{\mathcal{B}}
\nc\cC{\mathcal{C}}
\rnc\cD{\mathcal{D}}
\nc\cE{\mathcal{E}}
\nc\cF{\mathcal{F}}
\nc\cG{\mathcal{G}}
\rnc\cH{\mathcal{H}}
\nc\cI{\mathcal{I}}
\nc{\cJ}{\mathcal{J}} 
\nc\cK{\mathcal{K}}
\rnc\cL{\mathcal{L}}
\nc\cM{\mathcal{M}}
\nc\cN{\mathcal{N}}
\nc\cO{\mathcal{O}}
\nc\cP{\mathcal{P}}
\nc\cQ{\mathcal{Q}}
\rnc\cR{\mathcal{R}}
\nc\cS{\mathcal{S}}
\nc\cT{\mathcal{T}}
\nc\cU{\mathcal{U}}
\nc\cV{\mathcal{V}}
\nc\cW{\mathcal{W}}
\nc\cY{\mathcal{Y}}
\nc\cX{\mathcal{X}}
\nc\cZ{\mathcal{Z}}

\nc{\dmo}{\DeclareMathOperator}
\rnc{\Re}{\operatorname{Re}}
\rnc{\Im}{\operatorname{Im}}
\rnc{\span}{\operatorname{span}}
\dmo{\rank}{rank}
\dmo{\End}{End}
\dmo{\Hom}{Hom}
\dmo{\Jac}{Jac}
\dmo{\Id}{Id}
\dmo{\Eu}{Eu}
\dmo{\Aut}{Aut}
\dmo{\CP}{\bC P^1}
\dmo{\Hur}{{Hur}}

\title{Totally geodesic submanifolds of Teichm\"uller space}
\author{Alex~Wright}
\thanks{Note: This is a post-publication version which corrects an error pointed out by Frederik Benirschke and makes a few other clarifications. The error was in a claim that $\Omega L$ is complex analytic. It is actually only clear that it is real analytic, but this issue can be avoided by making technical changes to the argument. No change is required to any of the main results.  An erratum is also available \cite{Err}. }
\begin{document}
\maketitle
\thispagestyle{empty}



\section{Introduction}

\bold{Main results.} Let $\cT_{g,n}$ and $\cM_{g,n}$ denote the Teichm\"uller and moduli space respectively of genus $g$ Riemann surfaces with $n$ marked points. The Teichm\"uller metric on these spaces is a natural Finsler metric that quantifies the failure of two different Riemann surfaces to be conformally equivalent. It is equal to the Kobayashi metric \cite{Roy}, and hence reflects the intrinsic complex geometry of these spaces. 

There is a unique holomorphic and isometric embedding from the hyperbolic plane to $\cT_{g,n}$ whose image passes through any two given  points. The images of such maps,  called Teichm\"uller disks or complex geodesics, are much studied in relation to the geometry and dynamics of Riemann surfaces and their moduli spaces.

A complex submanifold of $\cT_{g,n}$ is called totally geodesic if it contains a complex geodesic through any two of its points, and a  subvariety of $\cM_g$ is called totally geodesic if a  component of its preimage in $\cT_{g,n}$ is totally geodesic. Totally geodesic submanifolds of dimension 1 are exactly the complex geodesics. 

Almost every complex geodesic in $\cT_{g,n}$ has dense image in $\cM_{g,n}$ \cite{Ma2, V2}. We show that higher dimensional totally geodesic submanifolds are much more rigid. 

\begin{thm}\label{T:main}
The image in $\cM_{g,n}$ of a totally geodesic complex submanifold of $\cT_{g,n}$ of dimension greater than 1 is a closed totally geodesic subvariety of $\cM_{g,n}$. 
\end{thm}

One dimensional totally geodesic subvarieties of  $\cM_{g,n}$ are called Teichm\"uller curves. There are infinitely many Teichm\"uller curves in each $\cM_{g,n}$. We show that higher dimensional totally geodesic submanifolds are much more rare. 

\begin{thm}\label{T:finite}
There are only finitely many totally geodesic submanifolds of  $\cM_{g,n}$ of dimension greater than 1.
\end{thm}

\bold{Context.} One source of totally geodesic submanifolds of $\cM_{g,n}$ is covering constructions, see \cite[Section 6]{MMW} for a definition. The first example of a totally geodesic submanifold of dimension greater than 1 not coming from a covering construction was given in \cite{MMW}, and two additional examples  appear in \cite{EMMW}. These three examples are totally geodesic surfaces in $\cM_{1,3}, \cM_{1,4}$ and $\cM_{2,1}$ respectively. 

Work of Filip implies that any closed totally geodesic submanifold of $\cM_{g,n}$ is in fact a subvariety \cite{Fi2}. Any real submanifold of $\cT_{g,n}$ that contains the Teichm\"uller disk between any pair of its points must in fact be a complex submanifold.

The inclusion of a totally geodesic complex submanifold into Teichm\"uller space must be an isometry for the Kobayashi metrics. Antonakoudis has shown that there is no holomorphic isometric immersion of a bounded symmetric domain of dimension greater than 1 into Teichm\"uller space \cite{St1}, and that any isometry of a complex disk into Teichm\"uller space is either holomorphic or antiholomorphic \cite{St2}. 

\bold{Elements of the proofs.} If $N$ is a subset of moduli or Teichm\"uller space, define $QN$ to be the locus of quadratic differentials which generate Teichm\"uller disks contained in $N$.  Typically $N$ will be a totally geodesic subvariety or submanifold, in which case we may view $QN$ as the cotangent bundle to $N$. Note that $QN$ is stratified according to the number of zeros and poles of the quadratic differential.

For every quadratic differential on a Riemann surface, either the quadratic differential is the square of an Abelian differential, or  there is a unique  double cover on which the lift of the quadratic differential is the square of an Abelian differential. The double cover is equipped with an involution. We call the Abelian differential together with this choice of involution the square root of the quadratic differential.  

Let $\Omega N$ be the locus of square roots of quadratic differentials in the largest dimensional stratum of $QN$. 
The following ingredient in our analysis may be of independent interest.  

\begin{thm}\label{T:trans}
If $N$ is a totally geodesic subvariety of moduli space, then  $\Omega N$ is transverse to the isoperiodic foliation.
\end{thm}

 Theorem  \ref{T:trans} is equivalent to saying that there is no nonconstant path in $\Omega N$ along which absolute periods of the Abelian differentials are locally constant. See \cite{McM:iso} for a definition of the isoperiodic foliation, which is also known as the kernel foliation, the absolute period foliation, and the rel foliation.

The proof of Theorem \ref{T:trans} uses results on cylinder deformations from \cite{Wcyl} and a classical result on Jenkins-Strebel differentials. 
Theorem \ref{T:finite} follows from Theorem \ref{T:trans} and recent finiteness results of Eskin-Filip-Wright \cite{EFW}. 

The proof of Theorem \ref{T:main} also uses Theorem \ref{T:trans} and results of \cite{EFW}. A key tool is the computation of the algebraic hull of the Kontsevich-Zorich cocycle from  \cite{EFW}. 



\bold{Acknowledgements.} This paper was inspired by comments of Curt McMullen, who in particular suggested the possibility that Theorem \ref{T:main} might be true. The author thanks Ben Dozier, Alex Eskin, Simion Filip, Steve Kerckhoff, Vlad Markovic, Curt McMullen,  Ronen Mukamel, and Mike Wolf for helpful conversations. 
 
This research was carried out in part at the ICERM conference ``Cycles on Moduli Spaces, Geometric Invariant Theory, and Dynamics", and was conducted during the period while the author served as a Clay Research Fellow. 

\section{Proof of Theorems \ref{T:finite} and \ref{T:trans}} 

We use notation consistent with \cite{MMW}. We assume some familiarity with recent results on the $GL(2, \bR)$ action on the Hodge bundle.  

If $N$ is a totally geodesic subvariety of moduli space, $\Omega N$ is an example of an affine invariant submanifold;  these are subvarieties of a stratum of $\Omega \cM_{g'}$ (for some $g'>0$) that are locally equal to a finite union of subspaces defined over $\bR$ in period coordinates \cite{EMM, Fi2}. The tangent space $\Omega N(Y, \omega)$ to an affine invariant submanifold $\Omega N$ at a point $(Y,\omega)$ is a subspace of relative cohomology $H^1(Y, \Sigma, \bC)$, where $\Sigma$ is the set of zeros of $\omega$. 
Let $p$ denote the map from relative to absolute cohomology.  The rank is defined to be half the dimension of $p$ of the tangent space \cite{Wfield}. This is an integer because $p$ of the tangent space is symplectic \cite{AEM}. 

To prove Theorem \ref{T:trans} we will compare the dimension of $\Omega N$ to that of $N$, using the following two results to get a lower bound on the dimension of $N$.  

An affine subspace is any translation of a vector subspace.   A Jenkins-Strebel differential is an Abelian or quadratic differential that is the union of horizontal cylinders and their boundaries; these are also known as horizontally periodic  differentials. 
Unless specified other, all references to (co)dimensions will be over $\bC$.

\begin{thm}\label{T:JS}
Any affine invariant submanifold $\Omega N$ of rank $r$ contains a set of Jenkins-Strebel differentials whose image in local period coordinates is a open subset of an  affine subspace $S$ of  codimension $r$, such that circumferences of horizontal cylinders are constant on this subset. 
 
 The  affine subspace $S$ is the translate of a subspace $L$ such that $p(L)$ is a Lagrangian in $p(\Omega N(Y, \omega))$ and such that $\ker(p)\cap \Omega N(Y, \omega) \subset L$. This $L$ is the subspace of $\Omega N(Y, \omega)$ which vanishes on the core curves of all horizontal cylinders. 
 \end{thm}

Theorem \ref{T:JS} can be viewed as a black box coming from \cite{Wcyl}, however we  provide  specific references to \cite{Wcyl}.

\begin{proof}
\cite[Theorem 1.10]{Wcyl} asserts the existence of a horizontally periodic $(Y,\omega)\in N$ such that the core curves of the horizontal cylinders span a subspace of the dual space of $\Omega N(Y, \omega)$ of dimension $r$. The subspace $L$ is the subspace of $\Omega N(Y, \omega)$ that annihilates all these core curves. Deforming $(Y,\omega)$ in any direction in $L$, the periods of the core curves of the horizontal cylinders remain constant. Hence all the horizontal cylinders of $(Y,\omega)$ persist on any such sufficiently small deformation, and remain horizontal and of constant circumference. 

The proof of \cite[Theorem 1.10]{Wcyl} in \cite[Section 8]{Wcyl} gives that for the $(Y,\omega)$ that is specially chosen in the proof, any sufficiently small deformation of $(Y,\omega)$ in the direction in $L$ does not create any new cylinders. Indeed,  \cite[Section 8]{Wcyl} gives that any such  deformation can be obtained by certain cylinder deformations of the horizontal cylinders of $(Y,\omega)$. Thus these deformations remain Jenkins-Strebel. 

The proof of \cite[Theorem 1.10]{Wcyl} gives that $p(L)$ is Lagrangian. Since $L$ has codimension $r$ and $p(L)$ has dimension $r$, it follows that $\ker(p)\cap \Omega N(Y, \omega) \subset L$. 
\end{proof}

Problems on the existence and uniqueness of Jenkins-Strebel differentials have been extensively studied, see for example \cite{Gardiner, HM, Jenkins, Liu, Strebel, Wolf}. Here we require only the following  uniqueness result. See for example Theorem 20.3 and the remarks after Lemma 20.3 in \cite{Strebel} for an expository account of the argument.

\begin{lem}\label{L:unique}
Let $X \in \cM_{g,n}$ be a Riemann surface. If two Jenkins-Strebel differentials $q, q'$ on $X$ have the same core curves of cylinders, and corresponding cylinders have the same circumference, then $q=q'$. 
\end{lem} 


 For the convenience of the reader, we include a proof.
\begin{proof}
Let cylinders on $q$ be $C_i$, and the corresponding cylinders on $q'$ be $C_i'$. Let their circumferences be $c_i$, and their heights be $h_i$ and $h_i'$ respectively. 

Note that $\sqrt{|q|}$ locally gives the flat  metric on $q$, and $|q|$ the flat area form. If we pick coordinates $x+iy$ on $C_i'$ so $|q'|=dxdy$ and $q=f(x,y)dx dy$ and $C_i'$ is identified with $[0,c_i]\times [0,h_i]$ (with an edge identification to glue the rectangle into a cylinder), then we get 
$$c_i \leq \int_0^{c_i} \sqrt{|f(x,y)|} dx$$
for any $y\in [0, h_i']$, with equality if and only if the horizontal circle in $C_i'$ at height $y$ is also a horizontal circle in $C_i$. This is true because  horizontal closed trajectories of $C_i$ are geodesics for the metric $\sqrt{|q|}$. Hence 
$$c_i h_i' \leq \int_{C_i'} \sqrt{|f(x,y)|} dx dy.$$

Now we use Cauchy-Schwarz with functions $\sqrt{|f(x,y)|}$ and $1$ to get 
$$(h_i' c_i)^2 \leq \left(\int_{C_i'} \sqrt{|f(x,y)|} dx dy \right)^2 \leq h_i' c_i \int_{C_i'} |f(x,y)| dx dy=  h_i' c_i \int_{C_i'} |q|.$$

Dividing by $h_i' c_i$ and summing we get  
$$\sum h_i' c_i \leq \sum  \int_{C_i'} |q|.$$
Because $q'$ is Jenkins-Strebel, the $C_i'$ cover the whole surface, so $\sum \int_{C_i'} |q|$ is  the area of $|q|$. Since the area of $|q|$ is $\sum h_i c_i$, we get 
$$\sum h_i' c_i \leq \sum h_i c_i.$$
By symmetry (reversing the roles of $q$ and $q'$), we get equality here. Hence $q$ and $q'$ have the same area, and so they have the same norm for the Teichm\"uller Finsler metric on the cotangent space to $\cM_{g,n}$. The same argument gives that any convex combination of $q$ and $q'$ has the same norm. This forces $q=q'$ since the unit ball of the Teichm\"uller metric is strictly convex. 
\end{proof}

A point of $\Omega {N}$ consists of a translation surface $(Y,\omega)$ and an involution $J$ that negates $\omega$, such that $Y/J\in {N}$. There is a map from $\Omega {N}$ to  $Q {N}$, because $\omega^2$ defines a quadratic differential on $Y/J$. In turn there is a forgetful map from $Q {N}$ to $\cM_{g,n}$ obtained by forgetting the quadratic differential but remembering the location of the poles. We will refer frequently to the composite of these two maps, which gives a map $\Omega {N}$ to $\cM_{g,n}$. For notational simplicity we will omit $J$ from our notation; there is no harm for our arguments in assuming it is the only involution on $Y$ negating $\omega$, as our arguments would not be any different were this not to be the case.

\begin{proof}[Proof of Theorem \ref{T:trans}]
Suppose $N$ has dimension $d$. Since $N$ is totally geodesic, there is a $d-1$ dimensional family of complex geodesics in $N$ passing through each point of $N$, so we get that $QN$ has dimension $2d$. Hence $\Omega N$ also has dimension $2d$. 

Let $r$ be the rank of $\Omega N$. By definition rank is at most half the dimension of $\Omega N$, so $r\leq d$. By Theorem \ref{T:JS} there is a $2d-r$ dimensional family of Jenkins-Strebel differentials in $\Omega N$, and hence also $QN$, with constant circumferences. By Lemma \ref{L:unique}
we see that the dimension of $N$ is at least $2d-r$. The inequalities $2d-r\leq d$ and $r\leq d$ give $r=d$. By definition of rank, it follows that the projection of the tangent space of $\Omega N$ to absolute cohomology has the same dimension as $\Omega N$. Since leaves of the isoperiodic foliation are tangent to the kernel of this projection, we get that $\Omega N$ is transverse to the isoperiodic foliation. 
\end{proof}

\begin{proof}[Proof of Theorem \ref{T:finite} given Theorem \ref{T:trans}.]
It is proved in \cite{EFW} that each stratum of Abelian differentials contains at most finitely many affine invariant submanifolds of rank at least 2. By Theorem \ref{T:trans}, if $N$ is a totally geodesic submanifold of dimension at least 2, then $\Omega N$ has rank at least 2. 

 Since $\Omega N$ determines $N$, and there are a finite list of strata that may contain $\Omega N$ for $N$ a totally geodesic submanifold of $\cM_{g,n}$, the result follows.  
\end{proof}

\section{Proof of Theorem \ref{T:main}}

This section requires the results and arguments from the previous section.  

Let $\tilde{N}$ be a totally geodesic submanifold of  $\cT_{g,n}$ of (complex) dimension $d>1$. Let $N$ denote the projection of $\tilde{N}$ to moduli space. Let $Q\overline{N}$ be the closure of $QN$. Note that $QN$  and  $Q\overline{N}$ are $GL(2,\bR)$ invariant. The goal of this section is to show that $Q\overline{N}=QN$, which implies $N$ is closed and hence establishes Theorem \ref{T:main}. In order to find a contradiction, we assume $Q\overline{N}\neq QN$. By \cite{EMM}, each stratum of $Q\overline{N}$ is an affine invariant submanifold. Since $QN$ is properly contained in $Q\overline{N}$, we see that $Q\overline{N}$ must have real dimension strictly greater than $4d$. 

The rough idea of the proof of Theorem \ref{T:main} is to consider all tangent spaces of totally geodesic submanifolds of dimension $d$ through each point of $\overline{N}$. Some version of this gives an equivariant subvariety of a Grassmanian bundle. Using \cite{EFW} we wish to show this subvariety is very large, so roughly speaking there are totally geodesic submanifolds of $\overline{N}$ through every point and  in so many directions that we are able to obtain a contradiction. 
The first step is to show that there is at least one totally geodesic submanifold through each point of $\overline{N}$. 

\begin{lem}\label{L:limitT}
Suppose that $L_k$ are totally geodesic submanifolds of $\cT_{g,n}$ of constant dimension, and that $X_k\in L_k$ converge to $X$. Let $P_k$ denote the cotangent space to $L_k$ at $X_k$, and suppose that $P_k$ converge to a subspace $P$ of the cotangent space of  $\cT_{g,n}$ at $X$. Then there is a totally geodesic submanifold of $\cT_{g,n}$ that passes through $X$ and whose cotangent space at $X$ is $P$. 
\end{lem}

\begin{proof}
Let $L$ be the set of all limit points of sequences $Y_k$ with $Y_k\in L_k$. If $\lim Y_k$ and $\lim Y_k'$ are two such points of $L$, then since the complex geodesic from $Y_k$ to $Y_k'$ lies in $L_k$, we get that the complex geodesic from  $\lim Y_k$ to $\lim Y_k'$ lies in $L$. (This can for example be proven as in the last paragraph of this proof.) 

Let $Q_{1}\cT_{g,n}$ be the bundle of quadratic differentials over $\cT_{g,n}$ of norm less than 1. There is a well known continuous map  $E:Q_{1}\cT_{g,n}\to \cT_{g,n}$ that maps $(Y,q)$ to the unique Riemann surface $Y'$ such that there is a Teichm\"uller mapping $Y\to Y'$ with initial quadratic differential $q$ and stretch factor $(1+\|q\|)/(1-\|q\|)$.  The restriction of $E$ to the quadratic differentials of norm less than 1 on any fixed Riemann surface is a  homeomorphism to $\cT_{g,n}$. See for example  \cite[Chapter 11]{FM} for a review of this material.

Since $L_k$ is totally geodesic, it contains the image of $P_k$ under $E$. By invariance of domain, this image is a real manifold of real dimension equal to the real dimension of $P_k$, so we see that $L_k$ is equal to the image of $P_k$. 

The restriction of $E$ to the preimage in $Q_{1}\cT_{g,n}$ of any compact subset of $\cT_{g,n}$ (under the standard projection $Q_{1}\cT_{g,n}\to \cT_{g,n}$) is a proper map. Hence we get that $L$ is the image of $P$ under $E$. By invariance of domain, $L$ is a real manifold of dimension equal to the real dimension of $P$. Since $L$ is totally geodesic, $L$ must in fact be a complex submanifold.   
\end{proof}

\begin{cor}\label{C:limit}
For every $(X, q)\in Q\overline{N}$ there is at least one $d$ dimensional   totally geodesic submanifold $L$ such that the Teichm\"uller disk generated by $q$ is contained in $L$ and $Q L  \subset Q\overline{N}$. 
\end{cor} 

Note that $L$ is not assumed to be closed (a priori it may be dense in  $\overline{N}$), and it is not assumed to be unique.

\begin{lem}\label{L:linear}
Let $L$ be a totally geodesic submanifold of $\cM_{g,n}$. We assume $L$ is complete but not  closed. Then $\Omega L$ is locally a countable union of  subsets that are  real analytic in period coordinates. 
\end{lem}

\begin{proof}
It is equivalent to show that the intersection of $QL$ with a stratum is locally real analytic. (By definition, $\Omega L$ lies in a single stratum.) A quadratic differential $(X,q)$ is in $QL$ if and only if $g(X,q)$ is in $L$ (after forgetting the quadratic differential) for all $g\in GL(2,\bR)$. Since the $GL(2,\bR)$ action on each stratum is real analytic and $L$ is complex analytic, this gives that $QL$ is locally defined by an (a priori infinite) set of real analytic equations. 
\end{proof}

We remark that  we get a countable union because $L$ may not be closed, so a neighbourhood of a point may contain countably many ``slices" of $L$. Each ``slice" is real analytic. We also remark that one should keep in mind that the map $q \to \overline{q}/|q|$ sending quadratic differentials to Beltrami differentials is not complex analytic. 

Consider the bundle over $\Omega \overline{N}$ whose fiber over a point $(Y, \omega)$ consists of all real subspaces of the real vector space $\Omega \overline{N}(Y,\omega)$ of real dimension $4d$ that contain $\span_\bC(\Re(\omega),\Im(\omega))$. We emphasize that $\Omega \overline{N}(Y,\omega)$ is of course a complex vector space, but we are choosing to view it as a real vector space, and that our subspaces of real dimension $4d$ are not necessarily complex subspaces. We let $\widehat{R}$ be the total space of this bundle, and denote fibers by $\widehat{R}(Y,\omega)$.

Given $V\in \widehat{R}(Y,\omega)$ and $h\in GL(2,\bR)$, there is a natural real linear map $$h_*: \Omega\overline{N}(Y,\omega) \to \Omega\overline{N}h(Y,\omega),$$ which is obtained by using the flat connection to view both domain and range as subspaces of $H^1(Y,\Sigma, \bR)\otimes \bC$ and letting $h\in GL(2,\bR)$ act trivially on the first factor and via its usual linear action on the second factor $\bC\simeq \bR^2$. 

For any real linear subspace $W$ of a complex vector space, we define $W_\bC =  W + i W$ to be the smallest complex vector subspace containing $W$. So in particular, 
$$(h_*V)_\bC= h_*V + i h_*V$$ is the smallest complex vector subspace of $\Omega \overline{N}h(Y,\omega)$ containing $h_*V$. 

We let $D\pi_{h(Y,\omega)}$ denote the derivative of the natural map $$\pi: \Omega \overline{N}\to \cM_{g,n}$$ at $h(Y,\omega)$. 

With these definitions in place, we now consider the subset $\widehat{R}'$ of $\widehat{R}$ consisting of those subspaces $V$ such that 
$$\dim_\bC D\pi_{h(Y,\omega)} (h_*V)_\bC \leq d$$
for all $h\in GL^+(2,\bR)$. Roughly speaking, we have $V\in \widehat{R}'(Y,\omega)$ if, from a certain limited point of view, $V$ might plausibly be associated with a $d$-dimensional  totally geodesic  submanifold of $\cM_{g,n}$ (like the $L$ in Corollary \ref{C:limit}).

We now define $R$ and $R'$ to be the sub-bundles of $\widehat{R}$ and $\widehat{R}'$ respectively where we additionally require the $V$ to be closed under taking real and imaginary parts. (These are the subspaces defined by real linear equations in suitable complex coordinates.)

\begin{lem}
$R'$ is closed in $R$, and $\widehat{R}'$ is closed in $\widehat{R}$. 
\end{lem}

\begin{proof}
This follows directly from the definition, since having  rank at most $d$ is a closed condition.
\end{proof}

\begin{lem}\label{L:inprime}
If $L$ is a totally geodesic complex manifold and $(Y,\omega)$ is a smooth point of $\Omega L$ and $V$ is the tangent space to $\Omega L$ at $(Y,\omega)$, then $(Y,\omega,V) \in \widehat{R}'$.
\end{lem}

\begin{proof}
This is immediate from the definition, because a map to a manifold of dimension $d$ can have rank at most $d$. Here we keep in mind that since $\Omega L$ is a real analytic variety of dimension $4d$, its tangent space at any smooth point is a real subspace of dimension $4d$.
\end{proof}

\begin{lem}\label{L:notempty}
Every fiber of $\widehat{R}'$ is nonempty. 
\end{lem}

\begin{proof}
This follows from Corollary \ref{C:limit} and Lemmas \ref{L:linear} and  \ref{L:inprime}, keeping in mind that the smooth points of a real analytic variety are dense. 
\end{proof}

Note that the definitions imply that $\widehat{R}'$ and ${R}'$ are $GL(2,\bR)$ invariant.

\begin{lem}\label{L:Zopen}
Let $U$ be a connected neighbourhood of a point $p\in \bC^a$, let $\cM$ be a complex manifold, and let $f: U\to \cM$ be complex analytic. For any $k$ and $d$, let $\cV_{k,d}$ be the set of real  subspaces $V$ of the tangent space to $\bC^a$ at $p$ of real dimension $k$ such that for all $h\in GL(2,\bR)$ in a neighbourhood of the identity, 
$\dim_\bC Df_{hp}(h_*V)_\bC\leq d.$ Then 
\begin{enumerate}
\item $\cV_{k,d}(p)$ is a subvariety of the Grassmanian of subspaces of $\bR^{2a}$ of dimension $k$, and 
\item in coordinates provided by the Pl\"ucker embedding, $\cV_{k,d}(p)$ is defined by a (possibly infinite) set of polynomials that vary real analytically with $p\in U$. 
\end{enumerate}
\end{lem}


\begin{cor}\label{C:var}
The fibers of $\widehat{R}'$ and $R'$ are real varieties. 
\end{cor}

\begin{proof}[Proof of Lemma \ref{L:Zopen}.]
For the moment we think of $p$ as fixed. Let $S$ be the set of $k$ planes through $p$ equipped with a choice of basis for the tangent space to the $k$ plane at $p$. It is equivalent to show that the set of $(V, v_1, \ldots, v_k)\in S$ for which 
$\dim_\bC Df_{hp}(h_*V)_\bC\geq d+1$
for arbitrarily small $h$ is a Zariski open subset of $S$. 

We may assume $\cM=\bC^b$. Using the basis $h_* v_1, \ldots, h_*v_k$ for $h_*V$, we may consider $Df_{hp}$ restricted to $h_*V$ as a  (complex) matrix  whose entries are real analytic functions on a neighborhood $V$ of the identity in $GL(2,\bR)$. If $\dim_\bC Df_{hp}(h_*V)_\bC\geq d+1$, then there is is some $d+1$ by $d+1$ minor of this matrix whose determinant $\rho$ is not identically zero.


Since $\rho$ is nonzero, there is some $\ell$ so that the $\ell$-th multivariate Taylor polynomial $\rho_\ell$ of $\rho$ centered at the identity in $V\subset GL(2,\bR)$ is also nonzero. Each coefficient of $\rho_\ell$ can be viewed as a polynomial function on $S$. (This polynomial depends on all  partial derivatives of order at most $\ell+1$ of $f$ at $p$. Since $f$ and $p$ are fixed, all these numbers may be viewed as constants.)
Let $c:S\to \bC$ be one of the nonzero coefficients of $\rho_\ell$. 

$(V, v_1, \ldots, v_k)$ is contained in the Zariski open set defined by $c\neq 0$. On this set, $\rho_\ell \neq 0$ and hence $\rho\neq 0$, and hence the rank of the derivative of $f$ restricted to an appropriate $k$-plane is at least $d+1$ at some point.  This proves the first statement. 

For the second statement, note that $\rho_\ell$ is an analytic function of $p$. 
\end{proof}

\begin{lem}\label{L:new}
Every fiber of $R'$ is non-empty. 
\end{lem}

\begin{proof}
Consider a point $(Y,\omega)$ whose orbit is dense in $\overline{N}$ and a $W\in \widehat{R}'(Y,\eta)$. For clarity, assume $(Y,\omega)$ is Lyapunov generic, although a weaker assumption suffices. Write $$W_\bC = \span_\bC(\Re(\omega), \Im(\omega))\oplus Q,$$ where the two factors are symplectically orthogonal after applying $p$. Let $g_t=\operatorname{diag}(e^t, e^{-t})$ denote Teichm\"uller geodesic flow. Since $\lambda_2<1$, the Hodge norms of the imaginary parts of vectors in $Q$ get contracted under $(g_t)_*$. In contrast, the Hodge norms of the real parts of vectors in $Q$ grow. Thus, it is helpful to write $Q=Q_{\mathrm{im}}\oplus Q_{\mathrm{good}}$, where $Q_{\mathrm{im}}$ is the subspace of vectors with zero real part and $Q_{\mathrm{good}}$   is any real subspace of $Q$ which is a complement. Note that there exists an $\epsilon>0$ such that  every non-zero vector $w$ in $Q_{\mathrm{good}}$ has $\|\Re(w)\| >  \epsilon \|\Im(w)\|.$

For every $(Y', \omega')\in \overline{N}$, we can find a sequence $t_n\to \infty$ such that $g_{t_n}(Y,\omega)\to (Y', \omega')$. We can also assume that the subspaces $(g_{t_n})_* W$ converge to a subspace $Z$ in $\widehat{R}'$, using compactness. Similarly, we can assume that the subspaces $(g_{t_n})_* Q_{\mathrm{good}}$ converge to a real subspace $Z_{\mathrm{good}}$ of $Z$. The Hodge norm comments above imply that every vector in $Z_{\mathrm{good}}$ has zero imaginary part. 

By definition, $Q$ is a complex subspace of complex dimension at least $2d-2$. Hence $Q_{\mathrm{good}}$ has real dimension at least $2d-2$, and of course the limit $Z_{\mathrm{good}}$ has the same dimension as $Q_{\mathrm{good}}$. Let $Z'$ be a subspace of $Z_{\mathrm{good}}$ of real dimension $2d-2$. Let $V=\span_\bC(\Re(\omega'), \Im(\omega'))\oplus (Z'\otimes\bC)$. We see that $V\in R'$. 
\end{proof}

From now on we will use $r$ to denote the rank of $\Omega \overline{N}$, and $b$ will denote $\dim_\bC \ker(p) \cap \Omega \overline{N}(Y, \eta)$ for any $(Y,\eta)\in \Omega \overline{N}$. Thus $2r+b$ is the dimension of $\Omega \overline{N}$.

\begin{lem}\label{L:notall}
For almost every $(Y,\omega)\in \Omega \overline{N}$, the following holds: 

$R'(Y,\omega)$ does not contain any subspaces contained in $$\ker(p)+\span_\bC(\Re(\omega), \Im(\omega)).$$ 

If $2d\geq b+2$, then there is a subspace $V$ in $R(Y,\omega)$ such that 
\begin{enumerate} 
\item $V\notin R'(Y,\omega)$, 
\item $\ker(p) \cap \Omega \overline{N}(Y, \eta) \subset V$, and 
\item $p(V)$ can be either expressed as the sum of  $\span_\bC(\Re(\omega), \Im(\omega))$ and an isotropic subspace (if $2d\leq b+r+1$), or it contains a Lagrangian subspace of $p(\Omega \overline{N}(Y, \eta))$ (if $2d\geq b+r+1$). 
\end{enumerate}
\end{lem} 

We think of the lemma as concerning subspaces in $R(Y,\omega)$ where the symplectic form is ``as degenerate as possible" given the restriction of being in $R(Y,\omega)$. 

\begin{proof}
%
%
By Theorem \ref{T:JS} there exists $(Y_0,\omega_0)\in \Omega \overline{N}$ and an affine subspace $S$ of  $\Omega \overline{N}(Y_0,\omega_0)$ that contains $\omega_0$,  has dimension $r+b$, and has the property that if $S_0$ is a neighbourhood of $\omega_0$ in $S$, so $S_0$ can be viewed as a subset of the stratum, then all the surfaces in $S_0$ are horizontally periodic with corresponding cylinders of the same circumference. By Lemma \ref{L:unique}, the map $S_0$ to $\cM_{g,n}$ is injective. By the Constant Rank Theorem (a corollary of the Inverse Function Theorem), since the  map $S_0 \to \cM_{g,n}$ is injective, the derivative must have rank at least $r+b$ at some point.
Replacing $(Y_0, \omega_0)$ by a nearby point in $S_0$ if necessary, we may assume that the derivative of $S_0 \to \cM_{g,n}$ has rank $r+b$ at $(Y_0, \omega_0)$.  That is, that derivative is injective. 

Recall from the statement of Theorem \ref{T:JS} that $S=\omega_0+L$, where $L$ is the subspace of $\Omega \overline{N}(Y_0,\omega_0)$ which is zero on all core curves of horizontal cylinders. Let $\bC S=L+\bC\omega_0$ be the subspace spanned by $S$. Since $\omega_0 \notin L$, $\bC S$ has dimension $r+b+1$. The subspace $\bC S$ is seen to be closed under real and imaginary parts, because $L$ is and $\Im(\omega_0)\in L$.

Now, consider any $(Y,\omega)\in \Omega \overline{N}$ with dense orbit. To prove the first claim, suppose $R'(Y,\omega)$  contains a subspace $V$ contained in $\ker(p)+\span_\bC(\Re(\omega), \Im(\omega)).$ Using that $R'$ is closed and taking a limit, we  can obtain a subspace $V_0\in R'(Y_0, \omega_0)$ contained in $\ker(p)+\span_\bC(\Re(\omega_0), \Im(\omega_0)).$ 

Noting that $V_0 \subset \bC S$ gives a contradiction, because $V_0 \in R'(Y_0, \omega_0)$ implies the dimension of $D\pi_{(Y_0, \omega_0)}(V_0)$ is at most $d$, but $V_0 \subset \bC S$ implies $D\pi_{(Y_0, \omega_0)}(V_0)$ has dimension at least $\dim(V_0)-1=2d-1$, and we have assumed $d>1$. 

To prove the second claim, suppose that $R'(Y,\omega)$ contains every $V$ in $R(Y,\omega)$ such that 
\begin{enumerate} 
\item $\ker(p) \cap \Omega \overline{N}(Y, \eta) \subset V$, and 
\item $p(V)$ can be either expressed as the sum of  $\span_\bC(\Re(\omega), \Im(\omega))$ and an isotropic subspace (if $2d\leq b+r+1$), or it contains a Lagrangian subspace of $p(\Omega \overline{N}(Y, \eta))$ (if $2d\geq b+r+1$).
\end{enumerate}
 
Again by taking limits, it follows that $R'(Y_0,\omega_0)$ contains every $V_0$ in $R(Y_0,\omega_0)$ such that 
\begin{enumerate} 
\item $\ker(p) \cap \Omega \overline{N}(Y_0, \eta_0) \subset V_0$, and 
\item $p(V_0)$ can be either expressed as the sum of  $\span_\bC(\Re(\omega_0), \Im(\omega_0))$ and an isotropic subspace (if $2d\leq b+r+1$), or it contains a Lagrangian subspace of $p(\Omega \overline{N}(Y_0, \eta_0))$ (if $2d\geq b+r+1$).
\end{enumerate}
If $2d\leq b+r+1$, then  $\bC S$ contains such a subspace $V_0$, because $2d\geq b+2$, and we obtain a contradiction as before. If $2d> b+r+1$, then $\bC S$ is contained in such a subspace $V_0$. We obtain a contradiction by noting that $\dim D\pi_{(Y_0, \omega_0)}(V_0)\geq r+b$ since $\bC S$ is contained in  $V_0$, and $\dim D\pi_{(Y_0, \omega_0)}(V_0)\leq d$ since $V_0\in R'(Y_0, \omega_0)$. This implies $r+b\leq d$ and hence  $2r+b \leq 2d$, which contradicts our assumption that $2r+b>2d$. 
\end{proof}

We now give the result of \cite{EFW} that we will use, phrased in a way to suit our present purpose. It can be viewed as a black box.  Define $G(Y,\omega)$ to be the subgroup of $GL(\Omega \overline{N}(Y,\omega))$ that acts trivially on $\ker(p)\cap \Omega \overline{N}(Y,\omega)$, preserves the tautological plane $\span(\Re(\omega), \Im(\omega))$, and induces a symplectic linear transformation of $p(\Omega \overline{N}(Y,\omega))$.

\begin{rem} There exists a basis for  $\Omega \overline{N}(Y,\omega)$ beginning with  a basis for $\ker(p)\cap\Omega \overline{N}(Y,\omega)$ followed by $\Re(\omega), \Im(\omega)$ with respect to which $G(Y,\omega)$ can be informally specified as 
$$\left(\begin{array}{ccc}
I & 0 & * \\
0 &SL(2,\bR) & 0 \\ 
0 & 0 & Sp(2r-2,\bR)
\end{array}
\right),$$
where $r$ is the rank and $I$ is an identity matrix. 
\end{rem}

\begin{thm}[Eskin-Filip-Wright]
Let $\Omega \overline{N}$ be any affine invariant submanifold, and let $\cT$ be its tangent bundle. Let $V$ be a measurable equivariant vector subbundle of any tensor power construction of $ \cT$ and its dual. Then, for almost every $(Y,\omega)\in  \Omega \overline{N}$, the fiber $V(Y,\omega)$ is invariant under $G(Y,\omega)$.  
\end{thm}

Note by definition $\Omega \overline{N}(Y,\omega)$ is the fiber of $\cT$ at $(Y,\omega)$, and hence  any linear transformation of $\Omega \overline{N}(Y,\omega)$ induces a linear transformation of any tensor power of this vector space or its dual. 

\begin{cor}\label{C:inv}
Let $R'(Y,\omega)$ be a subvariety of $R(Y,\omega)$ for all $(Y,\omega)\in \Omega\overline{N}$ that is equivariant and  that is defined in  the Pl\"ucker embedding as the set of zeros of a (possibly infinite) set of polynomials that vary analytically. Then at almost every $(Y,\omega)$, the fiber $R'(Y,\omega)$ is invariant under $G(Y,\omega)$. 
\end{cor}

\begin{proof}[Proof of Corollary.]
Recall that the Pl\"ucker embedding of the Grassmanian of $2d$ dimensional subspaces in $\Omega \overline{N}(Y,\omega)$ maps each such subspace to a line in $\bP(\Lambda^{2d}  \Omega \overline{N}(Y,\omega))$. Degree $D$ homogeneous polynomials on this projective space are elements of the $D$-th symmetric power of the dual of $\Lambda^{2d}  \Omega \overline{N}(Y,\omega).$ Both exterior and symmetric powers of a vector space are subspaces of tensor powers of that vector space. 

Let $R'_D(Y,\omega)$ be the subvariety of $R(Y,\omega)$ defined by those homogeneous polynomials of degree $D$ that vanish on $R'(Y,\omega)$. On the complement of a invariant analytic subvariety, the span of these polynomials has constant dimension. We thus get that the equivariant subbundle defined by these polynomials is invariant under $G(Y,\omega)$, and hence  $R'_D(Y,\omega)$ is invariant under $G(Y,\omega)$. 

$R'(Y,\omega)$ is the intersection of all the $R'_D(Y,\omega)$. The intersection of $G(Y,\omega)$ invariant sets must be $G(Y,\omega)$ invariant. 
\end{proof}

\begin{lem}\label{L:no}
If $2d\leq b+2$, then for almost every $(Y,\omega)$, the fiber $R'(Y,\omega)$ contains a subspace of $\ker(p)+\span_\bC(\Re(\omega), \Im(\omega)).$

If $2d\geq b+2$, then for almost every $(Y,\omega)$, the fiber $R'(Y,\omega)$ contains all $V \in R(Y,\omega)$ such that 
\begin{enumerate} 
\item $\ker(p) \cap \Omega \overline{N}(Y, \eta) \subset V$, and 
\item $p(V)$ can be either expressed as the sum of  $\span_\bC(\Re(\omega), \Im(\omega))$ and an isotropic subspace (if $2d\leq b+r+1$), or it contains a Lagrangian subspace of $p(\Omega \overline{N}(Y, \eta))$ (if $2d\geq b+r+1$). 
\end{enumerate}
\end{lem} 

\begin{proof}
This follows from Corollary \ref{C:inv}, because any nonempty closed subset of $R(Y,\omega)$ invariant under $G(Y,\omega)$ must contain all such  subspaces $S'$. Since $R'$ is closed, if this is true almost everywhere then in fact it is true everywhere. 
%
%
%
%
\end{proof}

\begin{proof}[Proof of Theorem \ref{T:main}.]
%
Lemma \ref{L:notall} contradicts Lemma \ref{L:no}.
\end{proof}

\bibliography{mybib}{}
\bibliographystyle{amsalpha}
\end{document}